\documentclass[11pt]{amsart}
\usepackage{amsfonts,amscd,latexsym,amsmath,amssymb,enumerate,verbatim,amsthm,epsfig,alltt,color,tikz-cd,bm,hyperref}
\usepackage[shortlabels]{enumitem}

\newcommand{\Aut}{\mathrm{Aut}}

\newcommand{\Rea}{\mathbb{R}}
\newcommand{\Nat}{\mathbb{N}}
\newcommand{\Int}{\mathbb{Z}}

\newcommand{\Com}{\mathbb{C}}

\newcommand{\AAA}{\mathcal{A}}
\newcommand{\PP}{\mathcal{P}}
\newcommand{\FF}{\mathcal{F}}
\newcommand{\GG}{\mathcal{G}}

\newcommand{\B}{\mathbb{B}}

\newcommand{\TT}{\mathbb{T}}
\newcommand{\II}{\mathcal{I}}
\newcommand{\JJ}{\mathcal{J}}

\newcommand{\Lip}{\mathrm{Lip}}
\newcommand{\BLip}{\mathrm{BLip}}
\newcommand{\HLip}{\mathrm{HLip}}

\newcommand{\cons}{\mathtt{constants}}

\usepackage{aliascnt}

\theoremstyle{plain}

\newtheorem{theorem}{\bf Theorem}[section]

\newaliascnt{lemma}{theorem}
\newtheorem{lemma}[lemma]{\bf Lemma}
\aliascntresetthe{lemma}

\newaliascnt{proposition}{theorem}
\newtheorem{proposition}[proposition]{\bf Proposition}
\aliascntresetthe{proposition}

\newaliascnt{corollary}{theorem}

\aliascntresetthe{corollary}

\newaliascnt{fact}{theorem}
\newtheorem{fact}[fact]{\bf Fact}
\aliascntresetthe{fact}

\newaliascnt{claim}{theorem}

\aliascntresetthe{claim}

\theoremstyle{definition}

\newaliascnt{definition}{theorem}
\newtheorem{definition}[definition]{\bf Definition}
\aliascntresetthe{definition}

\newaliascnt{example}{theorem}
\newtheorem{example}[example]{\bf Example}
\aliascntresetthe{example}

\newaliascnt{remark}{theorem}
\newtheorem{remark}[remark]{\bf Remark}
\aliascntresetthe{remark}

\newaliascnt{problem}{theorem}

\aliascntresetthe{problem}

\newaliascnt{question}{theorem}

\aliascntresetthe{question}

\usepackage{cleveref}

\begin{document}
\title[Invariant pointwise closed subspaces of Lipschitz spaces]{Invariant pointwise closed subspaces of Lipschitz spaces and their preduals}

\author{Michal Doucha}
\address{Institute of Mathematics, Czech Academy of Sciences, Žitná 25, 115 67 Praha 1, Czechia}
\email{doucha@math.cas.cz}
\urladdr{https://users.math.cas.cz/~doucha/}

\thanks{The author was supported by the GA\v{C}R project 25-15366S and by the Czech Academy of Sciences (RVO 67985840).}

\begin{abstract}
Motivated by both the research on Lipschitz-free spaces and on Lipschitz harmonic functions on graphs, we study invariant pointwise closed subspaces of spaces of Lipschitz functions over graphs, with special emphasis on finitely generated groups as graphs. Such spaces form a natural class of $\text{weak}^*$-closed subspaces, that can be fully described in some cases, and hence have canonical quotient preduals of the corresponding Lipschitz-free spaces. 

We show that these spaces are described by finite local constraints; in the group case as Lipschitz solutions of systems of convolution equations. We describe and characterize their preduals via a universal property and show that whenever they contain a non-zero element with a $c_0$-gradient, then they contain $\ell_\infty$, and consequently their preduals contain a complemented copy of $\ell_1$.

A guiding question is whether this class of Lipschitz spaces and their preduals contains an infinite-dimensional reflexive Banach space. In this regard, the main result of the paper is the following dichotomy proved using abstract harmonic analysis. Every translation-invariant pointwise closed subspace of the Lipschitz space over $\Int^d$ is either finite-dimensional or it is non-separable --in particular, the corresponding predual is either finite-dimensional or non-reflexive.
\end{abstract}

\maketitle

\section{Introduction}
The motivation for this work comes from two areas that have had relatively little direct interaction: the theory of Lipschitz-free Banach spaces and their duals, on the one hand, and the analytic theory of harmonic functions on graphs, on the other.

\emph{Lipschitz-free spaces} are free objects in the category of Banach spaces with bounded linear operators over the category of (pointed) metric spaces with Lipschitz maps (preserving distinguished points). As such they can be characterized, up to linear isometry, by the following universal property. Given a pointed metric space $(M,d,0)$ there is a unique, up to linear isometry, Banach space $\FF(M)$ and an isometric embedding $\delta:M\to\FF(M)$, preserving $0$, such that for any Banach space $Z$ and any Lipschitz map $f:M\to Z$, preserving $0$, there is a unique linear operator $\tilde f:\FF(M)\to Z$, with $\|\tilde f\|=\Lip(f)$, such that the following diagram commutes:

\[
\begin{tikzcd}[column sep=large, row sep=large]
M \arrow[r, "\delta"] \arrow[dr, "f"'] & \FF(M) \arrow[d, "\widetilde f"] \\
& Z
\end{tikzcd}
\]
Setting $Z$ to be the scalar field, which is in the theory of Lipschitz-free spaces as well as in this paper, $\Rea$, one gets that the dual $\FF(M)^*$ is canonically identified with $\Lip_0(M)$, the space of real-valued Lipschitz functions vanishing at $0$ with the norm being the optimal Lipschitz constant, to which we refer as the \emph{Lipschitz space} over $M$. We refer the reader to the monograpj \cite{Weaver} for more information.\medskip

By definition, there is a natural connection between the non-linear geometry of the underlying metric space and the linear geometry of the corresponding Lipschitz-free Banach space. A great deal of research in Lipschitz-free spaces has therefore been devoted to finding various interactions with metric geometry as seen e.g. in \cite{AGPP22,SmTa23,Bi25,BaLuRu25,Lu26,CaCuDo19} and references therein for a small sample.
\medskip

A slight disadvantage of Lipschitz-free spaces is that the universal property makes them very large. Indeed, by \cite{CuDoWo16}, as long as $M$ is infinite, there is an isomorphic embedding $\ell_\infty\hookrightarrow \Lip_0(M)$, which is equivalent with $\ell_1$ being embedded as a complemented subspace of $\FF(M)$. In particular, Lipschitz-free spaces are never reflexive unless they are finite-dimensional. A natural direction is to seek `smaller versions' of Lipschitz-free spaces where some weaker form of the universal property is preserved. Naturally, such spaces are quotients of $\FF(M)$, or by duality, preduals of some $\text{weak}^*$-closed subspaces of $\Lip_0(M)$. Since  by \cite{CuDoWo16}, every separable Banach space is a quotient of any infinite-dimensional $\FF(M)$, the choice of a quotient, or dually $\text{weak}^*$-closed subspace, must be properly justified. Some natural $\text{weak}^*$-closed subspaces of Lipschitz spaces have been investigated. For instance if $S\subseteq\Lip_0(M)$ is $\text{weak}^*$-closed and additionally a linear sublattice, then $S$ is itself a Lipschitz space of the form $\Lip_0(N)$ for some $N$, see \cite[Chapter 6]{Weaver} (see also \cite{Mandal} for another recent example of natural $\text{weak}^*$-closed subspaces of Lipschitz spaces).\medskip

Here we propose to study a particular class of $\text{weak}^*$-closed subspaces of Lipschitz spaces which is motivated by the example of harmonic functions on graphs. We introduce it here in the special case when the graph is a finitely generated group $G$, which we now fix. Given any positive probability measure $\mu$ on $G$, i.e. an element of $S^+_{\ell_1(G)}$, the positive part of the sphere of $\ell_1(G)$, a function $f:G\to\Com$ is $\mu$-harmonic if \[f(g)=\sum_{h\in G} \mu(h)f(gh),\quad g\in G.\] We shall reserve the term `harmonic function' without specifying the measure $\mu$, when $\mu$ is finitely supported and uniformly distributed on a finite generating set of $G$. Traditionally, the most studied class of $\mu$-harmonic functions on groups and graphs is the class of \emph{bounded} $\mu$-harmonic functions as there is a direct and beautiful connection with random walks on graphs and groups, for which we refer to the recent monographs \cite{La-book, Ya-book}. Recently, more general classes of harmonic functions have been studied, one of them being the class of \emph{Lipschitz} ($\mu$-)harmonic functions, which has been made popular through Kleiner's new proof of Gromov's theorem on groups of polynomial growth \cite{Kle10} where Lipschitz harmonic functions play a major role.

A simple observation is that as long as the measure $\mu$ on $G$ is finitely supported, the space $\HLip_0^\mu(G)$ of Lipschitz $\mu$-harmonic functions vanishing at $1_G$ is a pointwise closed, in particular $\text{weak}^*$-closed, subspace of $\Lip_0(G)$. Therefore, it is a dual Banach space whose predual is a natural quotient of $\FF(G)$. Interestingly, it follows from \cite{Kle10} that, in the case when $\mu$ is the uniform probability measure on a finite symmetric generating set, this space and its predual can be finite-dimensional even when $G$ is infinite, as this is the case when $G$ has polynomial growth. On the other hand, a recent result from \cite{ABGK24} shows that $\HLip_0(\GG)$ is a non-zero Banach space as long as $\GG$ is a locally finite vertex-transitive graph.

Two desirable properties of harmonic functions on graphs are:
\begin{itemize}
    \item They are invariant under graph automorphisms. That is, if $\GG$ is a locally finite connected graph, $f:\GG\to\Rea$ is a harmonic function, and $\phi:\GG\to\GG$ is a graph automorphism, then $f\circ \phi$ is harmonic as well.
    \item They are closed under pointwise convergence.
\end{itemize}

It is the goal of this paper to abstract these two properties and lay foundations for the study of more general pointwise closed subspaces of $\Lip_0(\GG)$, where $\GG$ is a locally finite connected graph, that are invariant under (certain subgroups of) graph automorphisms. Since pointwise closedness is stronger than $\text{weak}^*$-closedness, the pointwise closed spaces considered below automatically possess canonical preduals, while retaining considerably more finite-dimensional/local structure. As we shall see, this makes such spaces closely resemble spaces of Lipschitz harmonic functions and connects the study of these spaces with methods coming from symbolic and algebraic dynamical systems.

Finally, another motivation comes from isometric group actions on Banach spaces. Due to the functoriality of the Lipschitz-free space construction, given a finitely generated group $G$, the action of $G$ on itself by left multiplication induces a metrically proper action of $G$ on $\FF(G)$ by affine isometries. This idea was discussed in \cite[Section 4]{CaCuDo19}, where it was asked, in particular, whether $\FF(G)\simeq\ell_1$ for every hyperbolic group $G$. This question was recently answered positively in \cite{Ga25}. As explained in \cite{CaCuDo19}, this has an interesting consequence for groups with property (T). It follows from \cite{Ga25} (see also \cite{DrMa23}) that there are hyperbolic groups with property (T) which therefore act metrically properly on a space isomorphic to $\ell_1$, although they have a fixed-point property for such actions on $L_1$-spaces (see \cite{BFGM07,BaGeMo12} for details about these properties). An analogous question of Shalom for spaces isomorphic to a Hilbert space remains open. Since the class of Banach spaces arising directly as Lipschitz-free spaces over finitely generated groups appears rather restricted, we suggest looking beyond the spaces $\FF(G)$ themselves. The preduals of invariant pointwise closed subspaces considered here provide a natural larger class of quotients of $\FF(G)$, and their invariance gives rise to canonical affine isometric actions of $G$. We do not pursue this direction in the present paper, but regard it as a motivation for further research.

\subsection{Main results}
The paper has two main goals. The first is to introduce and motivate for further research invariant pointwise closed subspaces of Lipschitz spaces over graphs, and their preduals, as a new class of Banach spaces. The second is to develop a useful machinery for their study which culminates in a strong dichotomy for grids $\Int^d$ as graphs.

Regarding the latter, we show, among other things.
\begin{enumerate}
    \item Every invariant pointwise closed subspace of $\Lip_0(\GG)$ is determined by local constraints. When $\GG$ is a finitely generated group, it can be described using kernels of convolution operators. For polycyclic groups, finitely many such operators always suffice (see \Cref{thm:forbidpatterns,thm:wstartinvariantdescription,prop:fin-pres}).
    \item We explicitly describe preduals of such invariant pointwise closed subspaces using group algebra elements and we characterize them, uniquely up to linear isometry, as Banach spaces satisfying a certain universal property (see \Cref{thm:predual}, \Cref{prop:univ-property}).
    \item Whenever an invariant pointwise closed subspace $X$ of $\Lip_0(\GG)$ contains a non-zero element with a $c_0$-gradient, then it contains $\ell_\infty$, thus the predual of $X$ contains a complemented copy of $\ell_1$ (see \Cref{thm:linftyembeds}).
\end{enumerate}

These results culminate in a strong dichotomy. To motivate it, we emphasize that a desirable outcome is to find an invariant pointwise closed subspace of a Lipschitz space which is infinite-dimensional and separable, or even reflexive. To this end, the main result of the paper is to show that this is impossible for finite-dimensional grids, i.e. for $\Int^d$, where $d\geq 1$.

\begin{theorem}[see \Cref{thm:mainresult}]
Every translation-invariant pointwise closed subspace of $\Lip_0(\Int^d)$, where $d\geq 1$, is either finite-dimensional or non-separable.
\end{theorem}
\section{Preliminaries}
\subsection{Graph-theoretic notions}
In this note, by a graph $\GG$ we always mean a pair $\GG=(V_\GG,E_\GG)$, where $V_\GG$ is a set of vertices and $E_\GG\subseteq V_\GG^2$ is a symmetric set of (directed) edges. That is, if $(v,w)\in E_\GG$, where $v,w\in V_\GG$ are not necessarily distinct vertices, then also $(w,v)\in E_\GG$. For $e\in E_\GG$, we denote by $e^{-1}$ its inverse. Moreover, we consider maps $\alpha:E_\GG\to V_\GG$ and $\omega:E_\GG\to V_\GG$ so that for every $e$ we have $e=(\alpha(e),\omega(e))$.

\begin{definition}
Given a graph $\GG$, by $d_\GG$, or just $d$ if there is no risk of confusion, we denote the \emph{graph metric} on $\GG$. That is, given $v,w\in V_\GG$ \[\begin{split}d(v,w):=\min\Big\{& n\in\Nat\cup\{0\}\colon \exists v_0=v,\ldots,v_n=w\in V_\GG\\ &\forall 0\leq i<n\; \big((v_i,v_{i+1})\in E_\GG\big)\Big\},\end{split}\] assuming such $n\in\Nat\cup\{0\}$ exists. Otherwise, we set $d(v,w)=\infty$.
\end{definition}

We denote by $\Aut(\GG)$ the group of all automorphisms of $\GG$. We see each element $\phi\in\Aut(\GG)$ as a bijection of $V_\GG$ that preserves $E_\GG$. In one occasion we will need a topology on $\Aut(\GG)$, which will be the pointwise convergence topology.\medskip

\noindent\textbf{Convention.} In the sequel, unless stated otherwise, all ambient graphs $\GG$ are assumed to be connected, locally finite, vertex-transitive, infinite, and equipped with a distinguished vertex 0.

Note in particular that on each such graph $\GG$ the metric is finite and proper.

\subsubsection{Finitely generated groups as graphs and metric spaces} A natural source of such graphs is finitely generated groups. Let $G$ be an infinite group generated by a finite subset $S\subseteq G$, which is assumed to be symmetric, i.e. $S^{-1}:=\{s^{-1}\colon s\in S\}=S$. We define a graph denoted by $\GG_G$, or by abusing the notation, simply by $G$, where $V_{\GG_G}:=G$ and \[E_{\GG_G}:=\{(g,gs)\colon g\in G, s\in S\}.\] Such a graph is connected, locally finite, and we naturally take the group unit as the distinguished vertex (which we shall denote both by $1_G$ and $0$ depending on the context).

\subsection{Lipschitz functions on graphs}
\begin{definition}
Let $\GG$ be a graph. By $\Lip_0(\GG)$ we denote the Banach space of all real-valued Lipschitz maps vanishing at the distinguished vertex $0\in V_\GG$ with the norm, denoted by $\Lip(\cdot)$, being the Lipschitz constant. That is, given $f\in\Lip_0(\GG)$, we have \[\Lip(f):=\sup_{v\neq w\in V_\GG}\frac{f(v)-f(w)}{d(v,w)}.\]
\end{definition}
We shall need a few more notions related to the space $\Lip_0(\GG)$ that we define below.

\begin{definition}
Given any real-valued function $f:V_\GG\to\Rea$ we denote by $\nabla f:E_\GG\to\Rea$ its \emph{gradient}, defined on the set of edges as follows: \[\nabla f(e):=f(\omega(e))-f(\alpha(e)),\quad e\in E_\GG.\]
\end{definition}
The following is a basic fact whose proof is left to the reader.
\begin{fact}
Given a real-valued $f:V_\GG\to\Rea$, we have that $f$ is Lipschitz if and only if $\nabla f\in \ell_\infty(E_\GG)$.
\end{fact}

\begin{definition}
Let again $\GG$ be a graph and $f\in\Lip_0(\GG)$. Given a subset $C\subseteq E_\GG$, we write $\Lip_C(f)$ for the value \[\sup_{e\in C} |\nabla f(e)|.\] Clearly, $\Lip(f)=\Lip_{E_\GG}(f)$.
\end{definition}

\subsubsection{Preduals of the spaces of Lipschitz functions}
Given a graph $\GG$, the Banach space $\Lip_0(\GG)$ is a dual space that has a canonical predual called the \emph{Lipschitz-free space} over $\GG$ and denoted by $\FF(\GG)$. We refer to \cite[Chapter 3]{Weaver} for more details where these spaces are called \emph{Arens-Eells spaces}. Following \cite{Weaver}, we denote by $M[\GG]$ the space \[\{x\in\Rea^{V_\GG}\colon \text{for all but finitely many }v\in V_\GG\;\; x(v)=0\;\text{and}\;\sum_{v\in V_\GG} x(v)=0\}\] of all \emph{molecules}. It is spanned by the elements $\{m_{u,v}\colon u,v\in V_\GG\}$ called \emph{elementary molecules}, where $m_{u,v}(u)=1$, $m_{u,v}(v)=-1$, and $m_{u,v}(w)=0$, for $w\in V_\GG\setminus\{u,v\}$, and it forms a dense linear subspace of $\FF(\GG)$.

\subsubsection{Topologies on spaces of Lipschitz functions}
Given a graph $\GG$ as above, $\Lip_0(\GG)$ is a Banach space, thus it is naturally equipped with the norm topology. Since it is, as mentioned in the introduction, also a dual space to the Lipschitz-free space $\FF(\GG)$, it is also equipped with a strictly coarser (provided that $\GG$ is infinite which is the implicit assumption) topology -- the $\text{weak}^*$-topology. However, there is an even coarser natural topology on $\Lip_0(\GG)$ -- the \emph{pointwise convergence topology}, where $f_n\to f$ pointwise, for $(f_n)_n\subseteq\Lip_0(\GG)$ and $f\in \Lip_0(\GG)$, if $f_n(v)\to f(v)$ for all $v\in V_\GG$. The $\text{weak}^*$ and pointwise convergence topologies coincide on bounded subsets of $\Lip_0(\GG)$ (see \cite[Theorem 2.37]{Weaver}), however pointwise convergence is in general a strictly weaker condition than convergence in the $\text{weak}^*$-topology. As a consequence, pointwise closed subspaces of $\Lip_0(\GG)$ form a proper subclass of $\text{weak}^*$-closed subspaces of $\Lip_0(\GG)$, which will be utilized in the sequel.

The following straightforward fact, whose proof is left to the reader, makes the advantage of pointwise closedness explicit by relating it to finitely supported elements, i.e. molecules, in the predual.
\begin{fact}
Let $\GG$ be a graph as above. A linear subspace $X\subseteq \Lip_0(\GG)$ is pointwise closed if and only if for every net $(f_i)_i\subseteq X$ and $f\in \Lip_0(\GG)$, if $\langle f_i,x\rangle\to \langle f,x\rangle$, for every molecule $x\in M[\GG]$, then $f\in X$.
\end{fact}

Tautologically, by definition, a linear subspace $X\subseteq\Lip_0(\GG)$ is $\text{weak}^*$-closed if the above fact holds true replacing molecules by arbitrary elements $x\in\FF(\GG)$.
\subsection{Actions on function spaces}
Given a set $X$, a function (here taken real-valued) $f:X\to\Rea$, and a bijection $\phi:X\to X$, we can apply $\phi$ to $f$ via precomposition, i.e. taking $f\circ\phi$. In the language of group actions, given a subgroup $H\leq \Aut(X)$ of bijections of $X$, there is a natural action of $H$ on $\Rea^X$, the vector space of real-valued functions on $X$ by shifting the domain. That is, for $h\in H$ and $f\in \Rea^X$, the action of $h$ on $f$ results in a function $h\bullet f$ defined by \[h\bullet f(x)=f(h^{-1}x),\quad x\in X.\]

Let $H\leq \Aut(\GG)$ be a subgroup of $\Aut(\GG)$. Usually, but not always, we shall require that $H$ still acts transitively on $V_\GG$, in which case we call $H$ a \emph{transitive subgroup}. If $\GG=G$ for some finitely generated group $G$, then for $H$ we typically take $G$ acting on itself by left translations seen as a subgroup of $\Aut(\GG)$.

Notice that the action `$\bullet$' of $H$ on $\Rea^{V_\GG}$ however does not leave $\Lip_0(\GG)$ invariant, since for $h\in H$ and $f\in\Lip_0(\GG)$ we do not necessarily have $h\bullet f(0)=0$.  We shall define a different action exactly for this purpose. Given $h\in\Aut(\GG)$ and $f\in\Lip_0(\GG)$ we set \[h\cdot f(x):=f(h^{-1}x)-f(h^{-1}0).\] Notice that the $\cdot$-action does preserve $\Lip_0(\GG)$ and is by linear isometries.

\begin{definition}
Given a graph $\GG$ and a (usually transitive) $H\leq \Aut(\GG)$, a linear subspace $X\subseteq\Lip_0(\GG)$ is \emph{$H$-invariant}, or just \emph{invariant} if $H$ is clear from the context (e.g. if $\GG=G$ for some finitely generated group, then $H$ is usually assumed to be $G$), if \[h\cdot X:=\{h\cdot f\colon f\in X\}=X,\quad \forall h\in H.\]
\end{definition}

\begin{example}
Let $\GG$ be a graph. The subspace $\BLip_0(\GG)$ of \emph{bounded Lipschitz functions vanishing at $0$} is $\Aut(\GG)$-invariant.
\end{example}

\begin{example}
Given a graph $\GG$, define \[\Lip_0^\infty(\GG):=\{f\in\Lip_0(\GG)\colon \nabla f\in c_0(E_\GG)\},\] which is an $\Aut(\GG)$-invariant subspace.
\end{example}
Since the intersection of invariant subspaces is an invariant subspace, following the notation from the previous examples, we get that $\BLip_0(\GG)\cap \Lip_0^\infty(\GG)$ is $\Aut(\GG)$-invariant.

This last subspace will appear later and we prove the following useful fact.
\begin{lemma}\label{lem:invariantinfinite}
Let $\GG$ be a graph and $H\leq\Aut(\GG)$ be a co-compact subgroup (i.e. $\Aut(\GG)/H$, or equivalently $H\backslash\Aut(\GG)$, is compact). Then every $H$-invariant subspace of $\Lip_0^\infty(\GG)$ is either trivial or infinite-dimensional.

In particular, $\Lip_0^\infty(\GG)$ is either trivial or infinite-dimensional.
\end{lemma}
Outside of the main body of the proof, we remark that the condition that $H$ is co-compact is equivalent to the condition that there exists a co-bounded set of vertices $V'\subseteq V_\GG$ such that for each $v\in V'$ there is $h\in H$ such that $h\cdot 0=v$, where a subset $V'\subseteq V_\GG$ is \emph{co-bounded} if there is $K>0$ such that for all $v\in V_\GG$ there is $v'\in V'$ with $d(v,v')\leq K$. Indeed, set $V':=\{h\cdot 0\colon h\in H\}$. Suppose first that $H$ is co-compact. If $V'$ is not co-bounded, then we can find a sequence $(B_n)_{n=0}^\infty$, where $B_n$ is a ball of radius $2^n$ around $v_n\in V_\GG$, $n\in\Nat\cup\{0\}$, and $v_0=0$, such that for every $n>0$ and $w\in B_n$ there are no $v\in \bigcup_{i=0}^{n-1} B_i$ and $h\in H$ such that $h\cdot v=w$. For each $v\in \bigcup_{n=1}^\infty B_n$ there is $g_v\in\Aut(\GG)$ such that $g_v\cdot 0=v$. By co-compactness of $H$, i.e. compactness of $H\backslash \Aut(\GG)$, there is a converging subsequence $(Hg_{v_n})_{n=1}^\infty$, where $(v_n)_n\subseteq\bigcup_{m=1}^\infty B_m$. In particular, for all large enough $n,m$ there is $h\in H$ such that $g_{v_n}^{-1}h^{-1}g_{v_m}\cdot 0=0$, i.e. $h\cdot v_m=v_n$, a contradiction.

Conversely, suppose $V'$ is co-bounded, witnessed by some constant $K$. Let $(g_n)_{n=1}^\infty\subseteq \Aut(\GG)$ be an arbitrary sequence, we show that $(Hg_n)_n$ has a converging subsequence in $H\backslash\Aut(\GG)$. For each $n$, let $h_n\in H$ be such that $d(g_n\cdot 0, h_n\cdot 0)\leq K$. Passing to a subsequence, if necessary, we may assume that there is $v\in V_\GG$, satisfying $d(0,v)\leq K$, such that for all $n$, $h_n^{-1}g_n\cdot 0=v$. It follows that for each $n,m$, $\big(h_m^{-1}g_m\big)^{-1}h_n^{-1}g_n\in \mathcal{K}:=\{g\in\Aut(\GG)\colon g\cdot 0=0\}$. Since $\mathcal{K}$ is compact, we can, by passing to a subsequence if necessary, assume that $(h_n^{-1}g_n)_n$ is convergent, which finishes the proof.
\begin{proof}
We fix a non-zero $H$-invariant subspace $X\subseteq \Lip_0^\infty(\GG)$ and let us suppose that it is finite-dimensional. Then there is a finite ball $0\in C\subseteq V_\GG$ such that for every $f\in X$, if $f\restriction C\equiv 0$, then $f=0$. Let $E_C\subseteq E_\GG$ be the edges incident with at least one vertex from $C$. It follows that $\Lip_{E_C}$ is an equivalent norm on $X$, i.e. there is $K>0$ such that for every $f\in X$
\begin{equation}\label{eq:lem:invariantinfinite}
    K\cdot\Lip(f)\leq \Lip_{E_C}(f)\leq \Lip(f).
\end{equation}
Pick $0\neq f\in X$. Since $\nabla f\in c_0(E_\GG)$, for every $\varepsilon>0$, for all but finitely many $e\in E_\GG$, $|\nabla f(e)|<\varepsilon$. However, then, since $H$ is co-compact, there exists a sequence $(h_n)_{n\in\Nat}\subseteq H$ such that $\Lip_{E_C}(h_n\cdot f)<1/n$, for every $n\in\Nat$. This contradicts \eqref{eq:lem:invariantinfinite} for large enough $n$.
\end{proof}

\subsubsection{Relation between the two actions} The action `$\cdot$' is sometimes more difficult to deal with than the simple shift action `$\bullet$'. There is, however, a simple remedy. Given a graph $\GG$, denote by $\cons_\GG$, or just $\cons$ if $\GG$ is clear from the context, the one-dimensional space of constant real-valued functions on $V_\GG$. We shall work with $\Lip_0(\GG)\oplus \cons_\GG$ (if the isometry type of the space is important, we always mean the $\ell_\infty$-sum, i.e. $\Lip_0(\GG)\oplus \cons$ is implicitly meant to be $\Lip_0(\GG)\oplus_{\ell_\infty} \cons$).

We also note that the term `invariant subspace' is reserved for a vector subspace invariant with respect to the $\cdot$'-action. Spaces invariant under the $\bullet$-action are called \emph{$\bullet$-invariant}.

\begin{proposition}\label{prop:invsp-corresp}
Given a graph $\GG$ and a subgroup $H\leq \Aut(\GG)$, there is a one-to-one correspondence between (pointwise/$\text{weak}^*$/norm-closed) invariant subspaces of $\Lip_0(\GG)$ and (pointwise/$\text{weak}^*$/norm-closed) $\bullet$-invariant subspaces of $\Lip_0(\GG)\oplus \cons$ containing $\cons$.
\end{proposition}
\begin{proof}
Let $X\subseteq \Lip_0(\GG)$ be a vector subspace invariant with respect to the $\cdot$-action of $H$. Consider $X\oplus\cons\subseteq \Lip_0(\GG)\oplus\cons$. Consider $x\oplus c\in X\oplus\cons$, where $x\in X$ and $c\in \cons$. Then \[h\bullet (x\oplus c)=h\cdot x\oplus (c+x(h^{-1}0))\in X\oplus \cons,\quad h\in H,\] since by the assumption $h\cdot x\in X$.

Conversely, let $\cons\subseteq Y\subseteq \Lip_0(\GG)\oplus\cons$ be invariant with respect to the $\bullet$-action. Let $X:=P[Y]$, where $P:\Lip_0(\GG)\oplus\cons\to\Lip_0(\GG)$ is the canonical projection. Pick $f\in X$ and $h\in H$. Since $\cons\subseteq Y$, $f\oplus c\in Y$, for every $c\in\cons$, in particular, $f\in Y$. Thus, \[h\cdot f=(h\bullet f)\oplus -f(h^{-1}0)=P(h\bullet f)\in X,\] showing that $X$ is invariant as $f$ and $h$ were arbitrary. Moreover, it is clear that $Y=X\oplus \cons$, thus the map is one-to-one.

It is plain that the map sends norm-closed subspaces to norm-closed subspaces and pointwise closed subspaces to pointwise closed subspaces. The same is true also for $\text{weak}^*$-closed subspaces since it follows from \cite[Theorem 2.37]{Weaver} these are norm-closed subspaces whose unit balls are closed under pointwise convergence.
\end{proof}

\section{Invariant pointwise closed subspaces of $\Lip_0(\GG)$}

In this section, we describe invariant pointwise closed subspaces of $\Lip_0(\GG)$ using methods from linear symbolic dynamics. We refer the reader to \cite[Chapter 8]{CeCo-book23} for more information on this subject.
\subsection{The convolution operator description}
Let $A\subseteq V_\GG$ be a non-empty finite subset. By $\Rea^A$ we denote the finite-dimensional real vector space isomorphic to $\Rea^{|A|}$. By a \emph{pattern}, adapting the notation from symbolic dynamics, we mean a vector subspace of $\Rea^A$, for some finite $A\subseteq V_\GG$. Given a pattern $P$, by the \emph{support} of the pattern $P$, we mean the finite subset, denoted $A_P\subseteq V_\GG$, such that $P\subseteq \Rea^{A_P}$.
\begin{definition}
Let $\GG$ be a graph, $H\leq \Aut(\GG)$ a subgroup, $X\subseteq\Lip_0(\GG)$ an $H$-invariant linear subspace, and $P\subseteq \Rea^{A_P}$ a pattern. Given $f\in\Lip_0(\GG)$, we say that $f$ is compatible with $P$  if $h\cdot f\restriction A_P\in P$, for every $h\in H$.
\end{definition}
The following is a useful way of describing invariant pointwise closed subspaces.
\begin{proposition}\label{thm:forbidpatterns}
Let $\GG$ be a graph and $H\leq \Aut(\GG)$ be a subgroup. An invariant linear subspace $X\subseteq\Lip_0(\GG)$ is a pointwise closed invariant subspace of $\Lip_0(\GG)$ if and only if there exists a set $\PP$ of patterns such that $X$ consists precisely of those elements $f\in\Lip_0(\GG)$ compatible with each $P\in\PP$.    
\end{proposition}

\begin{proof}
First, suppose that $\PP$ is a set (finite or infinite) of patterns such that the invariant linear subspace $X$ consists precisely of those elements $f\in\Lip_0(\GG)$ compatible with each $P\in\PP$. Denoting for each $P\in\PP$, by $L_P$ the set \[\{f\in\Lip_0(\GG)\colon \forall h\in H\; (h\cdot f\restriction A_P\in P)\},\] we have $X=\bigcap_{P\in\PP} L_P$. Clearly, each $L_P$ is an invariant pointwise closed linear subspace of $\Lip_0(\GG)$, thus $X$ is an invariant pointwise closed linear subspace as well.
\medskip

Conversely, suppose now that $X$ is an invariant pointwise closed linear subspace. Pick any $f\in\Lip_0(\GG)\setminus X$. Without loss of generality, suppose that $f\in B_{\Lip_0(\GG)}$, i.e. $\Lip(f)\leq 1$. We claim that there exists a pattern $P_f$ such that $f$ is not compatible with $P_f$, however, every element from $X$ is. As soon as the claim is proved, we are done since we can take $\PP=\{P_f\colon f\in\Lip_0(\GG)\setminus X\}$.

Let $(A_n)_{n\in\Nat}$ be an increasing sequence of non-empty finite subsets of $V_\GG$, i.e. $A_n\subseteq A_{n+1}$, for each $n$, satisfying $V_\GG=\bigcup_n A_n$. For each $n$, set $a_n:=f\restriction A_n\in \Rea^{A_n}$. We claim that there exists $m$ such that for no $f'\in X$, $f'\restriction A_m=a_m$. Indeed, if not, then for each $n$ there is $f'_n\in X$ such that $f'_n\restriction A_n=a_n$. Since $V_\GG=\bigcup_n A_n$, $(f'_n)_n$ converges pointwise to $f$, thus we obtain $f\in X$, which is a contradiction.

So we can fix $m$ as described above. Now, set \[P_f:=\{f'\restriction A_m \colon f'\in X \}\]  Clearly, $P_f$ is a vector subspace of $\Rea^{A_m}$. In fact, it is a pattern that is, by definition, compatible with every element of $X$.

On the other hand, we claim that $f$ is not compatible with $P_f$. Indeed, otherwise $a_m=f\restriction A_m\in P_f$, thus there is $f'\in X$ such that $f\restriction A_m=f'\restriction A_m$, which is a contradiction.  This finishes the proof.
\end{proof}

\subsection{The case of pointwise closed subspaces over groups}\label{sec:Groupcase} Suppose now that $\GG=G$, where $G$ is a finitely generated group. Denote by $\Rea[G]$ the \emph{real group algebra of $G$}, consisting of formal finite linear combinations of elements of $G$ and multiplication defined so that it extends the multiplication on $G$ and is distributive with respect to the formal addition. One can see $\Rea[G]$ as finitely supported functions in $\Rea^G$, by which we denote all functions from $G$ to $\Rea$.

Given $x\in\Rea[G]$ and $f\in \Rea^G$ we define their \emph{convolution} $f\star x\in\Rea^G$ as follows: \[f\star x(g):=\sum_{h\in G} f(gh^{-1})x(h),\quad g\in G.\] We define $x\star f$ symmetrically. Given $x\in \Rea[G]$, by $P_x:\Rea^G\to\Rea^G$ we denote the right multiplication operator. That is, $P_x(f):=f\star x$.

We also recall that $\Rea[G]$ is moreover a `${}^*$-algebra'. That is, it has a linear involutive operation ${}^*$ satisfying $(xy)^*=y^*x^*$, for all $x,y\in\Rea[G]$. It is defined by $g^*:=g^{-1}$, for every $g\in G$, and extended linearly to $\Rea[G]$.

\begin{definition}\label{def:tildespaces}
Given an invariant subspace $X\subseteq\Lip_0(G)$ we denote by $\tilde X$ the corresponding $\bullet$-invariant subspace of $\Lip_0(G)\oplus\cons$ as provided by \Cref{prop:invsp-corresp}.   
\end{definition}

\begin{theorem}\label{thm:wstartinvariantdescription}
Let $G$ be a finitely generated group and $X\subseteq\Lip_0(G)$ be an invariant linear subspace. Then $X$ is an invariant pointwise closed subspace if and only if there exist $(x_n)_n\subseteq \Rea[G]$ such that \[\tilde X=\big(\Lip_0(G)\oplus\cons\big)\cap \bigcap_n \mathrm{ker}(P_{x_n}),\] where $\mathrm{ker}(P_x):=\{f\in\Rea^G\colon P_x(f)=0\}$.
\end{theorem}

\begin{proof}
The `if' implication is clear since for each $x\in\Rea[G]$, $\mathrm{ker}(P_x)\cap \big(\Lip_0(G)\oplus\cons\big)$ is clearly a pointwise closed subspace, and if $\tilde X$ is pointwise closed, then so is $X$. We show the other implication.

Suppose that $X$ is pointwise closed. By \Cref{thm:forbidpatterns} there exists a set $\PP$ of patterns, which is without loss of generality countable, such that $X$ consists of those elements of $\Lip_0(G)$ compatible with each $P\in\PP$, i.e. 

\begin{equation}\label{thm:wstartinvariantdescription-eq:1}
X=\{f\in\Lip_0(G)\colon \forall P\in\PP\;\forall g\in G\; (g\cdot f\restriction A_P\in P)\}.
\end{equation}

For any $P\in\PP$, let $\tilde P$ be the vector subspace of $\Rea^{A_P}$ spanned by $P$ and the constant vector $c_P\in \Rea^{A_P}$ defined by $c_P(g)=1$, for every $g\in A_P$. It is clear that analogously to \eqref{thm:wstartinvariantdescription-eq:1} we have
\begin{equation}\label{thm:wstartinvariantdescription-eq:2}
\tilde X=\{f\in\Lip_0(G)\oplus\cons\colon \forall P\in\PP\;\forall g\in G\; (g\bullet f\restriction A_P\in \tilde P)\}.
\end{equation}

Next, for every $P\in\PP$ such that $\tilde P$ is a proper subspace of $\Rea^{A_P}$ we consider its non-trivial orthogonal complement $\tilde P^\perp$ with respect to the canonical inner product on $\Rea^{A_P}$. Pick any element $p\in \tilde P^\perp$ viewed here as an element of $\Rea[G]$ and define $x_p:=p^*$.

It is enough to check that for every $f\in\Lip_0(G)\oplus\cons$ we have
\begin{equation}\label{thm:wstartinvariantdescription-eq:3}
\forall g\in G\; (g\bullet f\restriction A_P\in \tilde P)\Leftrightarrow f\in\bigcap_{p\in \tilde P^\perp} \mathrm{ker}(P_{x_p}). 
\end{equation}
Fix $f\in\Lip_0(G)\oplus\cons$. We claim that for every $g\in G$, $g\bullet f\restriction A_P\in \tilde P$ if and only if for every $p\in \tilde P^\perp$, $P_{x_p}(f)(g^{-1})=0$, which shows \eqref{thm:wstartinvariantdescription-eq:3}. To prove the claim we additionally fix $g\in G$.  For every $p\in \tilde P^\perp$, we have \[\langle g\bullet f\restriction A_P,p\rangle=\sum_{t\in A_P} f(g^{-1}t)\cdot p(t)=\sum_{t\in A_P} f(g^{-1}t)\cdot x_p(t^{-1})=f\star x_p(g^{-1}).\] Therefore \[g\bullet f\restriction A_P\in \tilde P\Leftrightarrow \forall p\in \tilde P^\perp\; \big(\langle g\bullet f\restriction A_P,p\rangle =0\big)\Leftrightarrow \forall p\in \tilde P^\perp\; f\star x_p(g^{-1})=0,\] which shows \eqref{thm:wstartinvariantdescription-eq:3}.
\end{proof}

\begin{remark}
Notice that to verify \eqref{thm:wstartinvariantdescription-eq:3} from the proof of \Cref{thm:wstartinvariantdescription} it was enough to check that $f\in L_P$ if and only if $f\in \bigcap_{i=1}^n \mathrm{ker}(P_{x_{p_i}})$, where $p_1,\ldots,p_n$ is a basis of $\tilde P^\perp$, since $\bigcap_{i=1}^n \mathrm{ker}(P_{x_{p_i}})=\bigcap_{p\in \tilde P^\perp} \mathrm{ker}(P_{x_p})$.
\end{remark}

\begin{definition}
Let $G$ be a finitely generated group. Say that an invariant pointwise closed subspace $X$ of $\Lip_0(G)$ is \emph{finitely presented} or \emph{of finite type} if there exist finitely many $x_1,\ldots,x_n\in\Rea[G]$ such that $\tilde X=\big(\Lip_0(G)\oplus\cons\big)\cap\bigcap_{i=1}^n \mathrm{ker}(P_{x_i})$.

In general, we call the set $\PP\subseteq\Rea[G]$ such that $\tilde X=\big(\Lip_0(G)\oplus\cons\big)\cap\bigcap_{x\in\PP} \mathrm{ker}(P_x)$ a \emph{presentation} of $X$.
\end{definition}
The terminology `finitely presented' is borrowed from algebraic dynamical systems, which bear many similarities with invariant Lipschitz spaces and partially motivated our research in this direction. See the monographs \cite{Sch-book,KerrLi-book} and \cite[Definition 13.1]{KerrLi-book}. The terminology \emph{of finite type} comes from symbolic dynamics, in particular from the theory of linear subshifts (see e.g. \cite[Chapter 8]{CeCo-book23} and \cite{CeCoPh22})\medskip

Surprisingly, there are graphs $\GG$ for which every invariant pointwise closed subspace of $\Lip_0(\GG)$ is finitely presented. Recall that a finitely generated group $G$ is \emph{polycyclic} if it can be obtained from cyclic groups by the group extensions, i.e. there is a normal series $G_0:=1\trianglelefteq G_1\trianglelefteq\ldots\trianglelefteq G_n:=G$ such that each quotient $G_i/G_{i-1}$, for $1\leq i\leq n$, is cyclic. This class includes all finitely generated abelian groups or more generally, all finitely generated nilpotent groups. In particular, it contains $\Int^d$, for all $d\geq 1$, and the Heisenberg group. We refer to \cite{CeCoPh22} for more information.

A consequence of \cite[Corollary 1.4]{CeCoPh22} is the following proposition.

\begin{proposition}\label{prop:fin-pres}
Let $G$ be a finitely generated polycyclic group. Then every translation-invariant pointwise closed subspace $X$ of $\Lip_0(G)$ is finitely presented.
\end{proposition}
\begin{proof}
By \Cref{thm:wstartinvariantdescription}, $\tilde X=\big(\Lip_0(G)\oplus\cons\big)\cap\bigcap_n \ker(P_{x_n})$, for some $(x_n)_n\subseteq\Rea[G]$. Set $Y:=\bigcap_n \ker(P_{x_n})\subseteq \Rea^G$, which is a linear subspace closed with respect to the product topology that is invariant under the shift action of $G$. By \cite[Theorem 1.3 and Corollary 1.4]{CeCoPh22}, there are finitely many $p_1,\ldots,p_n\in\Rea[G]$ such that $Y=\bigcap_{i=1}^n \ker(P_{p_i})$. Consequently, since $\tilde X=Y\cap \Lip_0(G)\oplus\cons$, we get that \[\tilde X=\big(\Lip_0(G)\oplus\cons\big)\cap\bigcap_{i=1}^n \ker(P_{p_i}).\]  
\end{proof}

\section{The predual}
Given a graph $\GG$ and a graph automorphism $\phi:\GG\to\GG$, we have noticed that \[f\in\Lip_0(\GG)\to f\circ\phi-f(\phi(0))\in\Lip_0(\GG)\] is a linear isometry of $\Lip_0(\GG)$. Since it is easily checked to be $\text{weak}^*$-$\text{weak}^*$-continuous, it is an adjoint operator of a linear isometry on $\FF(\GG)$, which can be explicitly described as follows. Given $v,w\in V_\GG$, the map \[m_{v,w}\to m_{\phi(v),\phi(w)}\] uniquely extends to a surjective linear isometry of $\FF(\GG)$.

Consequently, each subgroup $H\leq \Aut(\GG)$ naturally acts on $\FF(\GG)$ by linear isometries, which is the `predual action' to the $\cdot$-action of $H$ on $\Lip_0(\GG)$. Analogously, the $\bullet$-action of $H$ on $\Lip_0(\GG)\oplus_{\ell_\infty}\cons$ is a dual action to an action of $H$ on $\FF(\GG)\oplus_{\ell_1}\Rea$, which is the isometric predual of $\Lip_0(\GG)\oplus_{\ell_\infty}\cons$. Notice that the latter action has $\FF(\GG)$ as an invariant subspace, when seen as a subspace of $\FF(\GG)\oplus_{\ell_1}\Rea$, and the two actions coincide on $\FF(\GG)$ - the predual action of the $\cdot$-action and the predual action of the $\bullet$-action restricted to $\FF(\GG)$. We therefore do not need to specify which action on $\FF(\GG)$ we consider.

Notice that the second summand in $\FF(\GG)\oplus_{\ell_1}\Rea$ is, however, not invariant.

\subsection{Preduals of invariant pointwise closed subspaces over groups}
We explicitly describe the predual of $X\subseteq\Lip_0(\GG)$ in the case when $X$ is pointwise closed and $\GG$ is a finitely generated group $G$ since this is the most interesting and tractable case, as witnessed by the results from \Cref{sec:Groupcase}.

\begin{lemma}
Let $G$ be a finitely generated group and $X\subseteq\Lip_0(G)$ be an invariant pointwise closed subspace. Let $(x_n)_n\subseteq\Rea[G]$ be from \Cref{thm:wstartinvariantdescription} so that $\tilde X=\big(\Lip_0(G)\oplus\cons\big)\cap\bigcap_n\ker(P_{x_n})$. Then for each $n$, we have $x_n\in M[G]$, thus $x_n$ can be identified with an element of $\FF(G)$.
\end{lemma}

\begin{proof}
By construction from the proof of \Cref{thm:wstartinvariantdescription}, each $x_n$ is of the form $x_p$, where $p\in \tilde P^\perp$ for some compatible pattern $P\in\PP$ for $X$. Since $\tilde P$, by definition, contains the constant vector $c_P\in \Rea^{A_P}$, we have \[\sum_{g\in G} x_p(g)=\sum_{g^{-1}\in A_P} x_p(g)=\sum_{g\in A_P} p(g)=\langle p,c_P\rangle=0,\] which finishes the proof.
\end{proof}

\begin{theorem}\label{thm:predual}
Let $\GG$ be a graph, $H\leq\Aut(\GG)$ be a subgroup, and $X\subseteq\Lip_0(\GG)$ be an $H$-invariant pointwise closed subspace. Then $X$ is a dual Banach space, whose predual is an $H$-invariant quotient of the Lipschitz-free space $\FF(\GG)$, denoted further $\FF_X(\GG)$.

If $\GG=G$ is a finitely generated group and $\PP\subseteq\Rea[G]$ is a presentation of $X$, then \[\FF_X(G)=\FF(G)/\overline{\mathrm{span}\{g\cdot x^*\colon x \in\PP, g\in G\}}.\]
\end{theorem}

\begin{proof}
Since $X$ is a pointwise closed subspace of a dual space $\Lip_0(\GG)$, in particular a $\text{weak}^*$-closed subspace, it is a dual space itself. The predual $\FF_X(\GG)$ is then $\FF(\GG)/X^\perp$, where $X^\perp$ is the pre-annihilator of $X$ in $\FF(\GG)$. Given any $z\in\FF(\GG)$, denote by $[z]$ its equivalence class in the quotient $\FF_X(\GG)=\FF(\GG)/X^\perp$. We define the action of $H$ on $\FF_X(\GG)$ by \[h\cdot [z]:=[h\cdot z],\quad h\in H,\; z\in\FF(\GG),\] where $h\cdot z$ denotes the canonical action of $H$ on $\FF(\GG)$. We need to check that it is well-defined, for which it suffices to check that $X^\perp$ is an $H$-invariant subspace of $\FF(\GG)$. That immediately follows from the fact that $X$ is $H$-invariant. Indeed, given $z\in X^\perp$, $h\in H$ we have \[h\cdot z\in X^\perp\Leftrightarrow \forall f\in X\;\big(f(h\cdot z)=0\big)\Leftrightarrow \forall f\in X\;\big(h^{-1}\cdot f(z)=0\big)\] and the right-hand side obviously holds true since $z\in X^\perp$ and $X$ is $H$-invariant.\medskip

Now we suppose that $\GG=G$, where $G$ is a finitely generated group. We only need to check that $\overline{\mathrm{span}\{g\cdot x^*\colon x\in\PP, g\in G\}}=X^\perp$. To verify the inclusion $\overline{\mathrm{span}\{g\cdot x^*\colon x\in\PP, g\in G\}}\subseteq X^\perp$, it is enough to show that for every $x\in\PP$, for every $g\in G$, and $f\in X$ we have $f(g\cdot x^*)=0$. Since $f(g\cdot x^*)=g^{-1}\bullet f(x^*)$, $f\in X\subseteq \tilde X$, and $\tilde X$ is $G$-invariant with respect to the $\bullet$-action, it suffices to check that $f(x^*)=0$. This immediately follows by realizing that \[f(x^*)=\sum_{g\in G} f(g)x^*(g)=\sum_{g\in G} f(g)x(g^{-1})=P_x(f)(1_G)=0.\]

We verify the reverse inclusion. Suppose on the contrary that there is $z\in X^\perp\setminus \overline{\mathrm{span}\{g\cdot x^*\colon x\in\PP, g\in G\}}$. By the Hahn-Banach theorem, there exists $f\in \FF(G)^*=\Lip_0(G)$ such that $f(z)=1$ and $f$ vanishes on $\overline{\mathrm{span}\{g\cdot x^*\colon x\in\PP, g\in G\}}$. However, then for every $x\in \PP$, and $g\in G$ \[P_x(f)(g)=f\ast x(g)=(g^{-1}\bullet f\ast x)(1_G)=(g^{-1}\bullet f)(x^*)=f(g\cdot x^*)=0,\] thus $f\in \mathrm{ker}(P_x)$. Since $x\in\PP$ were arbitrary, we get that $f\in \tilde X\cap \Lip_0(G)$, thus $f\in X$, which contradicts that $f(z)=1$ for $z\in X^\perp$.
\end{proof}
\subsection{The universal property}
Like the Lipschitz-free spaces themselves, the preduals of the invariant pointwise closed subspaces of $\Lip_0(\GG)$, where $\GG$ is an appropriate graph, can be characterized, up to linear isometry, via a certain universal property. For convenience in notation, we restrict our treatment to the case when $\GG$ is a finitely generated group $G$, which is now fixed. Fix additionally a pointwise closed invariant subspace $X\subseteq\Lip_0(G)$ and the corresponding predual $\FF_X(G)$. Let $(x_n)_n\subseteq \Rea[G]$ be the elements of the group algebra such that $\tilde X=\big(\Lip_0(G)\oplus\cons\big)\cap\bigcap_n \ker(P_{x_n})$ and denote by $Q:\FF(G)\to\FF_X(G)$ the canonical quotient map. 
\begin{proposition}\label{prop:univ-property}
For every Banach space $Z$ and every Lipschitz map $f:G\to Z$ satisfying $f(1_G)=0$ and $f\ast x_n=0$, for all $n$, there is a unique bounded linear operator $F:\FF_X(G)\to Z$ with $\|F\|=\Lip(f)$ so that the following diagram commutes.

\begin{tikzcd}[row sep=large, column sep=large]
	G \arrow[r, "\delta"] \arrow[rrd, "f"'] & \FF(G) \arrow[r, "Q"] & \FF_X(G) \arrow[d, "\exists! F", dashed] \\
	& & Z
\end{tikzcd}

Any Banach space satisfying the same universal property is linearly isometric to $\FF_X(G)$.
\end{proposition}
\begin{proof}
By the universal property of $\FF(G)$ there exists a unique bounded linear operator $\tilde F:\FF(G)\to Z$ such that $f=\tilde F\circ\delta$. We verify that $\tilde F$ vanishes on $\{g\cdot x_n^*\colon n\in\Nat,g\in G\}$, thus it factorizes through an operator, denoted $F$, defined on the quotient $\FF_X(G)$. Fix $n\in\Nat$ and $g\in G$. Then \[\tilde F(g\cdot x_n^*)=\sum_{h\in G} f(gh^{-1})x_n(h)=f\ast x_n(g)=0,\] which proves the claim. To check that $\|F\|=\Lip(f)$, first recall that by the universal property of $\FF(G)$ we have $\|\tilde F\|=\Lip(f)$. Since, as we have just proved, $\tilde F=F\circ Q$ and $\|Q\|=1$, we get \[\|\tilde F\|=\|F\circ Q\|\leq \|F\|\|Q\|=\|F\|.\] On the other hand, for every $z\in S_{\FF_X(G)}$ and $\varepsilon>0$ there is $z'\in \FF(G)$ such that $[z']=z$ and $\|z'\|\leq 1+\varepsilon$, thus \[\|F(z)\|=\|\tilde F(z')\|\leq (1+\varepsilon)\|\tilde F\|\] and since $z\in S_{\FF_X(G)}$ and $\varepsilon>0$ were arbitrary, we get $\|F\|\leq \|\tilde F\|$, thus \[\|F\|=\|\tilde F\|=\Lip(f)\] as desired.

The last statement is a standard consequence of the universal property.
\end{proof}

\begin{example}
\Cref{prop:univ-property} may appear rather abstract, or even somewhat artificial in the general setting. Let us illustrate that it is natural on a particular example. Let $G$ be a finitely generated group and $X\subseteq\Lip_0(G)$ be the space of Lipschitz harmonic functions (with respect to the uniformly distributed measure on the finite symmetric generating set $S$ of $G$) vanishing at $1_G$. Then the predual $\FF_X(G)$ is characterized, up to linear isometry, as the unique separable Banach space $E$ for which there is a non-expansive map $\iota:G\to E$ and such that for every Banach space $Z$ and every Lipschitz harmonic $Z$-valued function $f:G\to Z$ (i.e. $f(g)=\sum_{s\in S} \frac{f(gs)}{|S|}$, for every $g\in G$) vanishing at $1_G$ there is a bounded linear operator $F$ satisfying $\|F\|=\Lip(f)$ such that the following diagram commutes.

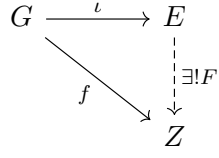
\begin{tikzcd}[row sep=large, column sep=large]
G \arrow[r, "\iota"] \arrow[rd, "f"'] & E \arrow[d, "\exists! F", dashed] \\
& Z
\end{tikzcd}  
\end{example}
We remark that Banach space-valued harmonic functions on groups were studied in \cite{JaNeu07}.
\subsection{Embeddings of $\ell_\infty$}
A strong way of demonstrating non-separability of a Banach space is to embed a copy of $\ell_\infty$ inside. It is known that for a dual space this is equivalent to containing a complemented copy of $\ell_1$ in the predual (see \cite[Theorem 4]{BePe58}). It was proved in \cite{CuDoWo16} that this happens for $\Lip_0(M)$, whenever $M$ is an infinite metric space.

Here we present a condition that allows to establish such embeddings for invariant pointwise closed subspaces of Lipschitz spaces.

\begin{theorem}\label{thm:linftyembeds}
Let $\GG$ be a graph, $H$ be a co-compact subgroup of $\Aut(\GG)$, and $X\subseteq\Lip_0(\GG)$ be a pointwise closed $H$-invariant subspace of $\Lip_0(\GG)$. If $X\cap\Lip^\infty_0(\GG)$ is non-zero, then $\ell_\infty$ embeds into $X$. In particular, $X$ is non-separable and $\FF_X(\GG)$ contains a complemented copy of $\ell_1$.
\end{theorem}

\begin{proof}
Let us denote $X\cap\Lip_0^\infty(\GG)$ by $X^\infty$. By \Cref{lem:invariantinfinite}, $X^\infty$ is infinite-dimensional, so for every finite ball $B\subseteq V_\GG$ around $0$ there exists a non-zero $f\in X^\infty$ such that $f\restriction B\equiv 0$.

Using that, one can clearly recursively define $1$-Lipschitz functions $(f_n)_{n\in\Nat}\subseteq X^\infty$ and strictly increasing balls $(B_n)_{n\in\Nat}$ around $0$ satisfying the following requirements, where $E_n\subseteq E_\GG$ is the set of edges whose both vertices belong to $B_n$, for each $n$.

\begin{enumerate}
    \item For every $n\in\Nat$, $\Lip(f_n)=\max_{e\in E_n} \nabla f_n(e)=1$.
    \item For every $e\in E_\GG\setminus E_n$, $|\nabla f_n(e)|<2^{-2n}$.
    \item For every $n\geq 2$, $f_n\restriction B_{n-1}\equiv 0$.\bigskip
\end{enumerate}

We define a map $T:\ell_\infty\to X$ as follows. For $x\in\ell_\infty$ we set 

\[T(x):=\sum_{n=1}^\infty x(n)f_n.\] 

Let us first check that $T(x)$ is indeed an element of $X$. For every $v\in V_\GG$, let $n\in\Nat$ be the least integer such that $v\in B_n$. Then $T(x)(v)=\sum_{i=1}^n x(i)f_i(v)<\infty$, thus $T(x)$ is a well-defined function from $V_\GG$ to $\Rea$, which moreover vanishes at $0$ as every $f_n$ vanishes there.

To check that it is Lipschitz, we verify that $\nabla T(x)\in\ell_\infty(E_\GG)$. Fix $e\in E_\GG$ and let $n$ be the smallest integer such that $e\in E_n$. Then for every $m<n$, $|\nabla f_m(e)|<2^{-2m}$, $|\nabla f_n(e)|\leq 1$, and for every $k>n$, $|\nabla f_k(e)|=0$. Thus for every $e\in E_\GG$ we have 
\begin{equation}\label{eq:Lipestimate}
|\nabla T(x)(e)|\leq \sum_{n=1}^\infty |x(n)\nabla f_n(e)|\leq \|x\|_{\ell_\infty}\sum_{n=0}^\infty 2^{-2n}<\infty.
\end{equation}

Next we check that $T(x)\in X$. By \Cref{thm:forbidpatterns}, there exists a set $\PP$ of patterns such that a Lipschitz function from $\Lip_0(\GG)$ is in $X$ if and only if it is compatible with each $P\in\PP$. We check that $T(x)$ is compatible with each $P\in\PP$. Fix $h\in H$ and $P\in\PP$. We verify that $h\cdot T(x)\restriction A_P\in P$. By definition, we have $h\cdot T(x)=\sum_{n=1}^\infty x(n)\,(h\cdot f_n)$. However, since $A_P$ is finite, there exists $n_0\in\Nat$ such that $h\cdot x(n')f_{n'}\restriction A_P=0$, for all $n'>n_0$. Therefore, $h\cdot T(x)\restriction A_P=\sum_{n=1}^{n_0} h\cdot x(n)f_n\restriction A_P$. Since for every $n\in\Nat$, $h\cdot x(n)f_n\restriction A_P\in P$, we get $h\cdot T(x)\restriction A_P\in P$ as well. It follows that $T(x)\in X$.

Finally, we check that $T$ is a linear isomorphic embedding. It is obviously linear. Fix any $n\in\Nat$ such that $|x(n)|\geq |x(m)|$, for all $m\leq n$, and let $e\in E_n$ be such that $\nabla f_n(e)=1$. Then since for every $m>n$, $\nabla f_m(e)=0$ and for every $k<n$, $|\nabla f_k(e)|<2^{-2k}$, we get
\begin{equation}\label{eq:Lipestimate2}
  |\nabla T(x)(e)|\geq |x(n)|-\sum_{i=1}^{n-1}|x(i)\nabla f_i(e)|\geq \frac{2|x(n)|}{3}.  
\end{equation}

It follows from \eqref{eq:Lipestimate} and \eqref{eq:Lipestimate2} that \[\frac{2}{3}\|x\|_{\ell_\infty}\leq \Lip(T(x))\leq \frac{4}{3}\|x\|_{\ell_\infty}.\]

The `In particular' part follows from \cite[Theorem 4]{BePe58}.
\end{proof}
\begin{example}
Let $\GG$ be an $n$-regular tree, for $n\geq 3$, i.e. a graph-theoretic tree where each vertex has exactly $n$ neighbors. Let $\HLip_0(\GG)$ be the space of all Lipschitz harmonic functions on $\GG$ vanishing at some distinguished vertex (root). Then $\ell_\infty\hookrightarrow \HLip_0(\GG)$.
\end{example}
\begin{proof}
For simplicity, we assume that $n=4$, so $\GG$ is the Cayley graph of a free group $F_2$ on two generators $a,b$. The general case is done by a straightforward modification. Let us define $f\in \HLip_0(\GG)$ as follows. We set $f(1_{F_2})=0$. Every non-trivial $w\in F_2$, which we now fix, is seen as a reduced word $w_1\ldots w_k$, for some $k\geq 1$, in the alphabet $\{a,b,a^{-1},b^{-1}\}$. We set \[f(w):=\begin{cases}
    \frac{a_k}{3^{k-1}} & \text{if }w_1\in\{a,b\},\\
    -\frac{a_k}{3^{k-1}} & \text{if }w_1\in\{a^{-1},b^{-1}\},
\end{cases}\]
where $(a_i)_{i=1}^\infty$ are defined by a recursive relation $a_i=3a_{i-1}+1$, for $i\geq 2$, with the initial value $a_1=1$. Then $f$ is harmonic which can be directly verified for $\{1_{F_2},a,b,a^{-1},b^{-1}\}$ and for $w\in F_2$ with $|w|:=k\geq 2$ we have \[\begin{split}f(w)&=\varepsilon \frac{a_k}{3^{k-1}}=\varepsilon\big(\frac{3a_k+1}{3^{k-1}}+\frac{a_k-1}{3^{k-1}}\big)/4=\varepsilon\big(3\frac{a_{k+1}}{3^k}+\frac{a_{k-1}}{3^{k-2}}\big)/4\\ &=\frac{f(wx)+f(wy)+f(wz)+f(ww_k^{-1})}{4},\end{split}\]
where $\varepsilon=1$ if $w_1\in\{a,b\}$ and $\varepsilon=-1$ if $w_1\in\{a^{-1},b^{-1}\}$, and $\{x,y,z\}=\{a,b,a^{-1},b^{-1}\}\setminus\{w_k^{-1}\}$. This verifies that $f\in\HLip_0(\GG)$.

Moreover, since $|\frac{a_k}{3^{k-1}}-\frac{a_{k-1}}{3^{k-2}}|=\frac{1}{3^{k-1}}$, we get that $\nabla f\in c_0(E_\GG)$. Thus by \Cref{thm:linftyembeds}, we get $\ell_\infty\hookrightarrow \HLip_0(\GG)$.
\end{proof}

\begin{remark}
As pointed out by C. Gartland \cite{Gartland-personal-communication}, the same is true for $\GG$ being any non-elementary hyperbolic group $G$; that is, $\ell_\infty\hookrightarrow\HLip_0(G)$ in this case.
\end{remark}

\begin{remark}
Lipschitz harmonic functions on groups with gradient in $c_0$ can be produced from strongly mixing unitary representations of groups, see \cite[Lemma 3.3]{Gou18} (and also \cite[Corollary 2.6]{GouJol16}). Therefore, again by \Cref{thm:linftyembeds}, $\ell_\infty\hookrightarrow\HLip_0(G)$ for such groups $G$.

On the other hand, an absence of a function $f$ in a general invariant pointwise closed subspace of $\Lip_0(\GG)$ with $\nabla f\in c_0(E_\GG)$ does not imply that $\ell_\infty\not\hookrightarrow X$. This follows e.g. from \cite{Gou-arxiv}, where it is shown that if $G$ is a direct product of two infinite groups, then there is no non-constant harmonic function on $G$ with a $c_0$-gradient, while it is easy to have such $G$ with $\ell_\infty\hookrightarrow\HLip_0(G)$. Even more directly, we refer to \Cref{ex:locked}, where a translation-invariant pointwise closed subspace of $\Lip_0(\Int^2)$ is produced which is isometric to $\ell_\infty$, yet it clearly cannot have a non-zero element with a $c_0$-gradient.
\end{remark}
\section{Invariant subspaces over finite-dimensional grids}
In the last section we solely focus on the case when the graph $\GG$ is the $d$-dimensional grid, for some $d\in\Nat$. Thus, as a finitely generated group, $\GG=\Int^d$. Much of the research in algebraic dynamical systems had been restricted to $\Int^d$ (see the monograph \cite{Sch-book}) prior the breakthrough in \cite{ChuLi15}. One reason is that the commutativity condition of the underlying group naturally connects this research with commutative algebra, real algebraic geometry, and harmonic analysis over locally compact abelian groups, which are well-developed subjects. This connection will naturally appear in our setting as well.

By $(e_i)_{i=1}^d$ we denote the canonical generators of $\Int^d$.

The following is well-known. We sketch the proof for the reader's convenience.
\begin{fact}
For every $d$, let $\HLip_0(\Int^d)$ be the space of Lipschitz harmonic functions on $\Int^d$, vanishing at $0$. Then $\HLip_0(\Int^d)$ is $d$-dimensional.
\end{fact}
\begin{proof}[Sketch of the proof] Fix $d\in\Nat$. Given any $f\in \HLip_0(\Int^d)$, consider its discrete partial derivative $(\partial_1 f,\ldots,\partial_d f)$, where $\partial_i f(x):=f(x+e_i)-f(x)$, for $i\leq d$. It is easy to check that for each $i\leq d$, $\partial_i f\in \big(\HLip_0(\Int^d)\oplus\cons\big)\cap \ell_\infty(\Int^d)$ (for more details, we forward-refer to the proof of \Cref{thm:mainresult}, where this is proved in bigger generality). However, since $\Int^d$ is abelian, $\big(\HLip_0(\Int^d)\oplus\cons\big)\cap \ell_\infty(\Int^d)$, which is the space of bounded harmonic functions on $\Int^d$, is equal to $\cons$  (see e.g. \cite{FrHaTaVa19}, where countable groups where this happens have been fully characterized). It follows that the map \[f\in \HLip_0(\Int^d)\to (\partial_1 f,\ldots,\partial_d f)\in \bigoplus_{i=1}^d \big(\HLip_0(\Int^d)\oplus\cons\big)\cap \ell_\infty(\Int^d)\] is a linear mapping from $\HLip_0(\Int^d)$ to $\Rea^d$. We leave to the reader the verification that it is bijective.
\end{proof}

\begin{remark}
It follows that the spaces of harmonic Lipschitz functions distinguish the dimension of the grid. We recall that it is a major open problem whether $\Lip_0(\Int^n)\simeq\Lip_0(\Int^m)$, for $n\neq m\geq 2$ (see \cite{CaCuDo19} for a discussion and context).
\end{remark}

The next example shows that the invariance and pointwise closedness conditions put a significant restriction on the spectrum of possible subspaces. It is governed by the geometry of the graph and the richness of its automorphism group which nevertheless also depends on the former. The case of $\Int$ happens to be degenerate. 

Recall that $\Lip_0(\Int)\equiv\ell_\infty$. In contrast, we have the following.

\begin{example}
Every proper translation-invariant pointwise closed subspace of $\Lip_0(\Int)$ is finite-dimensional.
\end{example}
\begin{proof}
Fix a proper translation-invariant pointwise closed subspace $X$ of $\Lip_0(\Int)$. By \Cref{thm:wstartinvariantdescription}, there are $(x_n)_n\subseteq \Rea[\Int]$ such that $\tilde X=\Lip_0(\Int)\oplus\cons\cap\bigcap_n \ker(P_{x_n})$. It suffices to show that $\tilde X$ is finite-dimensional and for this, it is in turn clearly sufficient to show that given any non-zero $x\in\Rea[\Int]$, $\ker(P_x)$ is finite-dimensional. Without loss of generality, assume that the support of $x$, the set $\{n\in\Int\colon x(n)\neq 0\}$, is contained in the interval $[0,m]$, for some $m\geq 0$, where $x(0)\neq 0$, $x(m)\neq 0$. Seeing $x$ as an element of $\ell_2([0,m])$, let $O$ be the orthogonal complement of the subspace spanned by $x$. Let $f\in \tilde X$. Then $(f(0),\ldots,f(-m))\in O$ since $P_x(f)(0)=\langle (f(0),f(-1),\ldots,f(-m)),x\rangle =0$. Notice that $f(1)$ is then fully determined as the solution of the linear equation $\langle (f(1),\ldots,f(-m+1)),x\rangle=0$. Analogously, $f(k)$ is uniquely determined for every $k\in \left(-\infty,-m-1\right]\cup \left[2,\infty\right)$. It follows that $\tilde X$ is at most $m$-dimensional as $O$ is $m$-dimensional.
\end{proof}

The spectrum of translation-invariant pointwise closed subspaces of $\Lip_0(\Int^d)$, where $d\geq 2$, is potentially much richer. For simplicity, we discuss here the case $d=2$.

First recall that $\Lip_0(\Int^2)\not\simeq\ell_\infty$ (this essentially follows from \cite{NaoSche07}; see additionally \cite{CaCuDo19} for more explanations). However, there are many ways to obtain a translation-invariant pointwise closed subspace of $\Lip_0(\Int^2)$ that is isomorphic, or even isometric, to $\ell_\infty$.

\begin{example}\label{ex:locked}
We define a subspace $X\subseteq\Lip_0(\Int^2)$ where we lock the second coordinate. That is, we set \[X:=\big\{f\in\Lip_0(\Int^2)\colon \forall v\in\Int^2\; \big(f(v+e_2)=f(v)\big)\big\}.\] It is plain to check that $X$ is a translation-invariant pointwise closed subspace of $\Lip_0(\Int^2)$ that is linearly isometric to $\ell_\infty$.

It is in some sense less interesting since the translation action of $\Int^2$ on $X$ factors through the action of $\Int^2/\langle e_2\rangle=\Int$ since $e_2$ acts trivially on $X$.
\end{example}

\begin{example}
Set \[X:=\big\{f\in\Lip_0(\Int^2)\colon \forall v\in\Int^2\;\big(f(v)-f(v+e_1)-f(v+e_2)+f(v+e_1+e_2)=0\big)\big\}.\] Again, $X$ is clearly a translation-invariant pointwise closed subspace. We claim that $X$ is linearly isometric to $\Lip_0(V)$, where $V:=\{(n,m)\in\Int^2\colon n=0\text{ or }m=0\}$. Indeed, first notice that there is a linear bijection $T:X\to \Lip_0(V)$, which is the restriction, i.e. $T(f):=f\restriction V$. This is clearly linear and it is bijective since each $g\in \Lip_0(V)$ has a unique extension to an element $\tilde g\in X$, which we let the reader verify. Clearly $\|T\|\leq 1$. We show that $\|T^{-1}\|\leq 1$. Pick $f\in X$, which is a unique extension of $f':=f\restriction V$, and an edge $e\in E_{\Int^2}$, which is of the form $(v,v\pm e_i)$, for $i\in \{1,2\}$. Without loss of generality, let us assume it is of the form $(v,v+e_1)$. We shall moreover assume that $v$ is in the quadrant $Q_1:=\{(n,m)\in\Int^2\colon n,m\geq 0\}$, as the other cases are analogous, and write $v$ as $(v_1,v_2)$. If $v_2=0$, then $|f(v+e_1)-f(v)|=|f'(v+e_1)-f'(v)|$, since both $v$ and $v+e_1$ are in $V$, thus $|f(v+e_1)-f(v)|\leq\Lip(f')$. Otherwise, set $v':=(v_1,v_2-1)$, so $v'\in Q_1$, and we have \[\begin{split}|f(v+e_1)-f(v)|&=|f(v'+e_2+e_1)-f(v'+e_2)|\\ &=|(f(v'+e_1)+f(v'+e_2)-f(v'))-f(v'+e_2)|\\ &=|f(v'+e_1)-f(v')|.\end{split}\] If $v_2-1=0$, then we again get \[|f(v+e_1)-f(v)|=|f(v'+e_1)-f(v')|=|f'(v'+e_1)-f'(v')|\leq\Lip(f').\] Otherwise, we continue analogously and we obtain that \[\begin{split}|f(v+e_1)-f(v)|&=|f((v_1,0)+e_1)-f((v_1,0))|=|f'((v_1,0)+e_1)-f'((v_1,0))|\\ &\leq\Lip(f')\end{split}\] which completes the proof that $T$ is an isometry.

Finally, since $V$ is a tree, it is clear that $\Lip_0(V)\equiv\ell_\infty$, thus $X\equiv\ell_\infty$. Moreover, $\Int^2$ clearly acts faithfully on $X$. In fact, it is not difficult to verify that $\Int^2$ acts metrically properly on the predual $\FF_X(\Int^2)\equiv\ell_1$.
\end{example}

However, we do not know of any example of a translation-invariant pointwise closed subspace of $\Lip_0(\Int^d)$ that would be infinite-dimensional and separable. It turns out this is not possible by the following strong dichotomy.

\begin{theorem}\label{thm:mainresult}
Fix $d\in\Nat$ and let $X\subseteq\Lip_0(\Int^d)$ be an invariant pointwise closed subspace. Then the following dichotomy holds.
\begin{enumerate}
    \item Either $X$ is finite-dimensional;
    \item or $X$ is non-separable.
\end{enumerate}
In particular, we have the following.
\begin{enumerate}[label=(\arabic*')]
    \item Either $\FF_X(\Int^d)$ is finite-dimensional;
    \item or $\FF_X(\Int^d)$ is not reflexive.
\end{enumerate}
\end{theorem}

\begin{proof}
First, notice that $X$ is finitely presented by \Cref{prop:fin-pres}. Thus, there are $p_1,\ldots,p_n\in\Rea[\Int^d]$ such that \[\tilde X=\big(\Lip_0(\Int^d)\oplus\cons\big)\cap\bigcap_{i=1}^n \ker(P_{p_i}).\]

We shall consider two commutative Banach algebras. We assume a basic knowledge of commutative Banach algebras in the proof, in particular, the existence of the Gelfand transform and Gelfand duality for commutative $C^*$-algebras.

From now on, we shall work with complex scalars. However, we continue assuming that $p_1,\ldots,p_n\in \Rea[\Int^d]$, i.e. their coefficients are real. Moreover, we still assume that $\tilde X=\big(\Lip_0(\Int^d)\oplus\cons\big)\cap \bigcap_{i=1}^n \ker(P_{p_i})$ is a real space. By $\tilde X^\Com$, we denote its complex version which is the intersection of the complex $(\Lip_0(\Int^d)\oplus\cons)^\Com$ with $\bigcap_{i=1}^n \ker(P_{p_i}^\Com)$. Since $p_1,\ldots,p_n$ are real, $\tilde X^\Com$ is in fact the complexification of $\tilde X$. In particular, $\tilde X\subseteq \tilde X^\Com$ both as sets and as real Banach spaces.

First, we consider the commutative $C^*$-algebra $C(\TT^d)$, equipped with pointwise addition, multiplication and the ${}^*$-operation, which is the pointwise complex conjugation. We remark that this algebra is, by the Gelfand transform, ${}^*$-isomorphic to the group $C^*$-algebra $C^*(\Int^d)$. 
Second, we consider the Banach ${}^*$-algebra $\ell_1(\Int^d)$, which is, besides the addition, equipped with
\begin{enumerate}
    \item multiplication in the form of convolution defined for $\phi,\psi\in\ell_1(\Int^d)$ by \[\phi\ast\psi(x):=\sum_{y\in \Int^d} \phi(x-y)\cdot\psi(y),\quad x\in\Int^d;\]
    \item the ${}^*$-operation, which is an involution defined for any $\phi\in\ell_1(\Int^d)$ by \[\phi^*(x):=\overline{\phi(-x)},\quad x\in\Int^d.\]
\end{enumerate}
The Gelfand space of $\ell_1(\Int^d)$, the space of all characters of $\ell_1(\Int^d)$ with the weak${}^*$-topology, is canonically homeomorphic to $\TT^d$. The Gelfand transform $\Phi:\ell_1(\Int^d)\to C(\TT^d)$ is defined \[\Phi(\delta_x)(t):=\prod_{i=1}^d t_i^{x_i},\quad \forall x=(x_i)_{i=1}^d\in\Int^d,\; t=(t_i)_{i=1}^d\in\TT^d,\] and extended linearly. It preserves the multiplication and the ${}^*$-operation.

For each $i\leq n$, set $\bm{p}_i:=\Phi(p_i^*)\in C(\TT^d)$ and let $\JJ\subseteq C(\TT^d)$ be the closed ideal generated by $\{\bm{p}_i\colon i\leq n\}$. Consider the quotient algebra $\B:=C(\TT^d)/\JJ$. By the Gelfand duality, $\B$ is isomorphic to $C(F)$, where $F\subseteq\TT^d$ is a closed subset. It is standard, since $\B=C(\TT^d)/\langle \bm{p}_i\colon i\leq n\rangle$, that we actually have \[F=\{x\in\TT^d\colon \forall i\leq n\;(\bm{p}_i(x)=0)\}.\]

We also let $\II\subseteq\ell_1(\Int^d)$ be the closed ideal generated by $\{p_i^*\colon i\leq n\}$ and denote by $\AAA$ the quotient algebra $\ell_1(\Int^d)/\II$. Notice that the Gelfand spectrum of $\AAA$ is $F$. Indeed, by definition, it consists of all the characters of $\ell_1(\Int^d)$ that vanish on $\II$. Since $\II$ is generated by $\{p_i^*\colon i\leq n\}$, it consists of all the characters of $\ell_1(\Int^d)$ that vanish on $\{p_i^*\colon i\leq n\}$. Then given a character $\chi$ represented by $f=(f_1,\ldots,f_d)\in\TT^d$ and $p_i^*$, for $i\leq n$, where $p_i^*=\sum_{j=1}^k \alpha_j v_j$, where $(\alpha_j)_{j=1}^k\subseteq\Rea$ and $(v_j)_{j=1}^k\subseteq\Int^d$, where further $v_j=(v_j^1,\ldots,v_j^d)\in\Int^d$ we have \[\chi(p_i^*)=\sum_{j=1}^k \alpha_j\big(\prod_{l=1}^d f_l^{v_j^l}\big)=\bm{p}_i(f),\] showing that characters vanishing on all $\{p_i^*\colon i\leq n\}$ coincide with elements of $\TT^d$ that annihilate all of $\{\bm{p}_i\colon i\leq n\}$, which is, by definition, the set $F$.

Next we set 
\[X_B:=\{f\in \ell_\infty(\Int^d)\colon \forall i\leq n\;(f\ast p_i=0)\}\subseteq \ell_\infty(\Int^d),\] which is a pointwise closed subspace of $\ell_\infty(\Int^d)$, and $X_B\subseteq \tilde X^\Com$.

We claim that there is a canonical identification of $\AAA^*$ and $X_B$. Indeed, elements of $\AAA^*$ are elements of $\ell_1(\Int^d)^*$ that annihilate $\II$, thus \[\begin{split}\AAA^* &=\{f\in\ell_1(\Int^d)^*\colon \forall z\in\II\; (f(z)=0)\}\\ &=\{f\in\ell_\infty(\Int^d)\colon \forall i\leq n\;(f\ast p_i=0)\}=X_B.\end{split}\]\medskip

\noindent\textbf{Claim 1.} If $\B$ is finite-dimensional, then $X$ and $\FF_X(\Int^d)$ are finite-dimensional as well.\smallskip

\noindent\textit{Proof of Claim 1.} First, we show that $X_B$ is finite-dimensional, for which it is therefore enough to show that the algebra $\AAA$ is finite-dimensional.

Since $\B=C(F)$ is finite-dimensional, it follows that $F$ is finite. Let us argue that $\AAA$ having finite Gelfand spectrum implies that $\AAA$ is finite-dimensional, in fact, $|F|$-dimensional (we emphasize this is not true for general commutative Banach algebras). 

Since $F$ is finite, each element of $F$ is relatively clopen in $F$. By the Shilov idempotent theorem, see \cite[Chapter II - Theorem 5]{BoDu-book}, for each $f\in F$ there is an idempotent $e_f\in \AAA$ whose Gelfand transform is the characteristic function of the singleton $\{f\}$. Notice that \[e_fe_{f'}=0,\quad \forall f\neq f'\in F,\quad \text{and }\sum_{f\in F} e_f=1\in\AAA.\] For the former, the Gelfand transform of $e_fe_{f'}$ is clearly zero and the Gelfand transform is injective on idempotents (notice that a product of idempotents in a commutative algebra is an idempotent). Similar argument gives the latter.  For $f\in F$, set $\AAA_f:=\AAA\cdot e_f$. Since $\AAA$ is commutative and $e_f$ is an idempotent, each $\AAA\cdot e_f$ is a subalgebra, which is closed as a range of the bounded projection $a\to a\cdot e_f$. It follows that $\AAA$ is a finite topological sum of the subalgebras $\AAA_f$, $f\in F$. Therefore, it suffices to check that each $\AAA_f$ is finite-dimensional.

Fix $f\in F$ and notice that the Gelfand spectrum of $\AAA_f$ is $\{f\}$. Indeed, since the projection $a\to a\cdot e_f$ is a bounded algebra morphism, each character $\chi$ of $\AAA_f$ can be, by precomposing with this projection, extended to a character of $\AAA$. Thus, characters of $\AAA_f$ are restrictions of characters of $\AAA$ and since the Gelfand transform of $e_f$ is the characteristic function of $\{f\}$, clearly, any character associated to $f'\neq f\in F$ vanishes on $\AAA_f$.

Thus, since $\AAA$ is a quotient of $\ell_1(\Int^d)$ and $\AAA_f$ is a quotient of $\AAA$, we get that $\AAA_f$ is a quotient of $\ell_1(\Int^d)$ with a spectrum consisting of a single point.

By \cite[Theorem 2.4.16]{ReiSte-book} applied to the image $\Phi[\ell_1(\Int^d)]$ of $\ell_1(\Int^d)$ in $C(\TT^d)$ by the Gelfand transform, we get that $\{f\}$, for each $f\in F$, is a set of spectral synthesis for $\Phi[\ell_1(\Int^d)]$, thus by the definition, \cite[Definition 2.4.14]{ReiSte-book}, $\AAA_f=\ell_1(\Int^d)/M_f$, where $M_f\subseteq\ell_1(\Int^d)$ is the maximal ideal such that the spectrum of $\ell_1(\Int^d)/M_f$ is $\{f\}$, which is the kernel of the character associated to $f$, thus $\AAA_f$ is $1$-dimensional.

Now, denote by $X^\Com$ the first summand of $\tilde X^\Com=(X\oplus\cons)^\Com$. Given any $f\in X^\Com$ and $i\leq d$ we denote by $\partial_i f$ the $i$-th partial discrete derivative, that is \[\partial_i f(x)=f(x)-f(x-e_i),\quad x\in\Int^d,\] where $e_i$ is the $i$-th canonical basis element of $\Int^d$.

We claim that $\partial_i f\in X_B$ for all $f\in X$ and $i\leq d$. Fix $f\in X^\Com$ and $i\leq d$. First, notice that $\partial_i f\in\ell_\infty(\Int^d)$. This follows since \[\|\partial_i f\|_{\ell_\infty}=\sup_{x\in\Int^d} |\partial_i f(x)|=\sup_{x\in\Int^d} |f(x)-f(x-e_i)|\leq \Lip(f)<\infty.\] Second, we verify that $\partial_i f\ast p_j=0$, for every $j\leq n$. However, since $\Int^d$ is abelian and so the convolution $\ast$ is commutative we get for every $j\leq n$ and $x\in\Int^d$ \[\partial_i f\ast p_j(x)=f\ast p_j(x)-f\ast p_j(x-e_i)=0-0=0,\] proving that $\partial_i f\in X_B$. Consider the map \[f\in X^\Com\to (\partial_1 f,\ldots,\partial_d f)\in \bigoplus_{i\leq d} X_B.\] The range is a finite sum of finite-dimensional spaces, so it is finite-dimensional, and the map is linear and injective. Indeed, linearity is obvious and injectivity follows from the fact that if $0\neq f\in X$ then there are $x,y\in \Int^d$ whose distance is $1$ in the canonical word metric on $\Int^d$, which is the `$\ell_1$-metric on $\Int^d$', such that $f(x)-f(y)\neq 0$, so for some $i\leq d$, $\partial_i f\neq 0$. It follows that $X^\Com$ is finite-dimensional itself, which in turn implies that $X$ is finite-dimensional. This finishes the proof of Claim 1.\medskip

\noindent\textbf{Claim 2.} If $\B$ is infinite-dimensional, then $X$ is non-separable.\smallskip

\noindent\textit{Proof of Claim 2.}

Since $\B=C(F)$ is infinite-dimensional, the compact subset $F\subseteq\TT^d$ is infinite. Identifying $\Com^d$ with $\Rea^{2d}$, we can view the zero set of each $\bm{p}_i$, $i\leq n$, as a real algebraic set. The same is true for $\TT^d$ seen as a subset of $\Rea^{2d}$, therefore $F$ is a compact infinite real algebraic set, thus in particular, an infinite semialgebraic set, so it must be uncountable, see \cite[Theorem 2.3.6]{BoCoRo-book} (see also \cite{BoCoRo-book} for the notions of algebraic and semialgebraic sets).

Since $F$ is uncountable, there is at least one coordinate $1\leq j\leq d$ such that the set $\{f_j\colon (f_i)_{i=1}^d\in F\}$ of the $j$-th coordinates of elements of $F$ is uncountable. Moreover, there exists $\lambda>0$ such that the set $\{f_j\colon (f_i)_{i=1}^d\in F, |1-f_j^{-1}|>\lambda\}$ is uncountable. Therefore, we can select an uncountable subset $F'\subseteq F$ whose $j$-th coordinates are all pairwise different and satisfy $|1-f_j^{-1}|>\lambda$, for every $(f_i)_{i=1}^d\in F'$. Finally, we can take a further, still uncountable, subset $F''\subseteq F'$ such that for every $f\neq t\in F''$, $f_jt_j^{-1}$ has infinite order, since for each $f\in F'$ the set $\{t\in F'\colon f_jt_j^{-1}\text{ has finite order}\}$ is at most countable.

Each $f\in F''$ corresponds to the character $\chi_f:\AAA\to\Com$ defined by \[\chi_f\big((x_i)_{i=1}^d\big):=\prod_{i=1}^d f_i^{x_i}.\] In particular, we have $\chi_f\in \AAA^*=X_B\subseteq \tilde X^\Com$. We shall show that there is a constant $K>0$ such that \[\Lip(\chi_f-\chi_t)>K,\quad f\neq t\in F'',\] thereby showing that $\tilde X^\Com\subseteq\big(\Lip_0(\Int^d)\oplus\cons\big)^\Com$, equipped with the $\big(\Lip\oplus_{\ell_\infty} |\cdot|\big)$-norm, is non-separable. Therefore, also $X$ with the Lipschitz norm is non-separable.

Fix $f\neq t\in F''$. Set $u:=(1-f_j^{-1})$ and $v:=(1-t_j^{-1})$. Since $\chi_f(0)=\chi_t(0)=1$, we have $\chi_f-\chi_t\in\tilde X^\Com\cap \Lip_0(\Int^d)^\Com$. We apply the partial derivative operator $\partial_j$. Let us estimate \begin{equation}\label{eq:Case2-1}
    \sup_{x\in\Int^d} |\partial_j \chi_f(x)-\partial_j \chi_t(x)|\leq \Lip(\chi_f-\chi_t)
\end{equation}
from below. For every $n\in\Int$, we denote by $x_n\in\Int^d$ the element whose $j$-th coordinate is equal to $n$ and all other coordinates are $0$. We have 
\begin{equation}\label{eq:Case2-2}
\begin{split}
|\partial_j \chi_f(x_n)-\partial_j \chi_t(x_n)| &=|\chi_f(x_n)-\chi_f(x_n-e_j)-(\chi_t(x_n)-\chi_t(x_n-e_j))|\\ & =|f_j^n-f_j^{n-1}-(t_j^n-t_j^{n-1})|=|f_j^n(1-f_j^{-1})-t_j^n(1-t_j^{-1})|\\ &=|(f_jt_j^{-1})^n u-v|.  
\end{split}
\end{equation}
 Since, by the assumption, $f_jt_j^{-1}$ has infinite order, $\{(f_jt_j^{-1})^n\colon n\in\Int\}$ is dense in $\TT$. Thus, since $\min\{|u|,|v|\}>\lambda$ we get, using \eqref{eq:Case2-1} and \eqref{eq:Case2-2}, \[\begin{split}\Lip(\chi_f-\chi_t) &\geq \sup_{x\in\Int^d} |\partial_j \chi_f(x)-\partial_j \chi_t(x)|\geq \sup_{n\in\Int} |\partial_j \chi_f(x_n)-\partial_j \chi_t(x_n)|\\ & =\sup_{n\in\Int} |(f_jt_j^{-1})^n u-v|>2\lambda,\end{split}\] so we are done with the constant $K=2\lambda$. This finishes the proof.
\end{proof}

\section*{Acknowledgements}
The author is grateful to Ebrahim Samei for a discussion about commutative Banach algebras in the proof of \Cref{thm:mainresult}; especially, for introducing the author to the sets of spectral synthesis and their use in the argument.

The author also thanks Christopher Gartland for a fruitful discussion about Lipschitz harmonic functions that motivated this research.

\noindent An \textbf{AI assistant}, namely ChatGPT Plus, was used in the final stage of the preparation of the manuscript for proofreading and for a help and reference with the basic algebraic geometry fact in Claim 2 in the proof of \Cref{thm:mainresult}.
\bibliographystyle{siam}

\bibliography{references}
\end{document}